\documentclass[12pt,reqno]{amsart} 
\usepackage{amsmath, amssymb, mathscinet,mathtools}
\usepackage{amsthm} 
\usepackage{enumitem}
\usepackage[T1]{fontenc}
\numberwithin{equation}{section}

\usepackage{xcolor}
\usepackage[
colorlinks=true,
linkcolor=blue,
citecolor=blue,
urlcolor=blue
]{hyperref}
\usepackage{mathrsfs} 

\usepackage{graphicx}
\usepackage{color}
\usepackage{comment}

\usepackage{titlesec}
\usepackage[nameinlink,noabbrev]{cleveref}

\titleformat{\section}{\normalfont\bfseries\large}{\thesection.}{0.5em}{}
\titleformat{\subsection}{\normalfont\bfseries}{\thesubsection.}{0.5em}{}
\titleformat{\subsubsection}{\normalfont\bfseries}{\thesubsubsection.}{0.5em}{}

\theoremstyle{plain} 
\theoremstyle{plain}
\newtheorem{theorem}{Theorem}[section]
\newtheorem{lemma}[theorem]{Lemma}
\newtheorem{corollary}[theorem]{Corollary}
\newtheorem{proposition}[theorem]{Proposition}

\theoremstyle{definition}
\newtheorem{definition}[theorem]{Definition}

\newcommand{\cH}{\mathcal H}

\newcommand{\ku}{\kappa_{\mathrm u}}

\newcommand{\cU}{\mathcal U}
\newcommand{\one}{\mathbf 1}
\newcommand{\PhiG}{\Phi}
\newcommand{\core}{\operatorname{core}}

\title{Non Uniform Kazhdan Constant for Linear Groups}

\author{Alexander Lubotzky}
\address{Department of Mathematics,
	Faculty of Mathematics and Computer Science,
	Weizmann Institute of Science,
	Rehovot 7610001, Israel}
\email{alex.lubotzky@mail.huji.ac.il}

\author{Jvbin Yao}
\address{Research Center for Operator Algebras,
	School of Mathematical Sciences,
	East China Normal University,
	Shanghai 200241, PR China}
\email{52285500009@stu.ecnu.edu.cn}

\dedicatory{Dedicated to Shahar Mozes on the occasion of his 65th birthday.}

\begin{document}
	
	\begin{abstract}

	   A group $\Gamma$ generated by a finite set $S$ has Property $(T)$ if the associated Kazhdan constant $\kappa(\Gamma,S)$ is positive. If so, the same holds for every finite generating set $S$. There are Property $(T)$ groups which are uniformly $(T)$, and others which are not.
	
		It is a well-known problem whether the classical examples of Property $(T)$ groups, $\mathrm{SL}_n(\mathbb Z)$, $n\geq 3$, are uniformly $(T)$ or not. We show that they are not. Moreover, the same holds for every infinite finitely generated linear group.

	\end{abstract}
	
	\maketitle

	\section{Introduction}
	
	Let $\Gamma$ be a finitely generated group and let $S\subseteq\Gamma$ be a
	finite set.  The Kazhdan constant of $\Gamma$ relative to $S$ is
	\[
	\kappa(\Gamma,S)
	=\inf_{\substack{\pi:\Gamma\to\mathcal U(\mathcal H)\\
			\mathcal H^\Gamma=\{0\}}}
	\ \inf_{\substack{\xi\in\mathcal H\\ \|\xi\|=1}}
	\ \max_{s\in S}\|\pi(s)\xi-\xi\|.
	\]
	If $\Gamma$ is finitely generated, its uniform Kazhdan constant is
	\[
	\kappa_{\mathrm u}(\Gamma)
	=\inf_{\substack{S\subseteq\Gamma\text{ finite}\\
			\langle S\rangle=\Gamma}}
	\kappa(\Gamma,S).
	\]
	A finitely generated group $\Gamma$ is said to have
	\emph{Property~$(T)$ uniformly} if $\kappa_{\mathrm u}(\Gamma)>0.$
	Notice that the infimum is taken over all finite generating sets, with no
	bound on their cardinalities.  Thus uniform Property~$(T)$ is considerably
	stronger than Property~$(T)$ itself, which only requires a positive Kazhdan
	constant for one finite generating set.
	
	The question of whether Property~$(T)$ can hold uniformly over all finite
	generating sets was raised by Lubotzky, in connection with the dependence of
	expansion and Kazhdan constants on generating sets; see
	\cite{LubotzkyWeiss,Lubotzky} and \cite[Section~8]{Shalom}.
	Gelander and \.{Z}uk showed that Property~$(T)$ does not imply uniform
	Property~$(T)$: they exhibited groups with Property~$(T)$ whose uniform
	Kazhdan constant is zero, and in particular showed that this phenomenon occurs
	for many lattices in Lie groups \cite{GelanderZuk}.  Osin subsequently proved
	that every infinite hyperbolic group has zero uniform Kazhdan constant
	\cite{Osin}.  In the opposite direction, Osin and Sonkin constructed infinite
	finitely generated groups with positive uniform Kazhdan constant
	\cite{OsinSonkin}.
	
	Although explicit Kazhdan constants are known for several natural generating
	sets of arithmetic groups \cite{Shalom,Kassabov}, these estimates do not
	address uniformity over all finite generating sets.

	A central concrete problem concerns $\operatorname{SL}_n(\mathbb Z)$. For $n\geq3$, these groups are among the most classical examples of groups with Property~$(T)$, but it remained open whether they have Property~$(T)$ uniformly.
	The question was posed explicitly in
	\cite[Problem~1.3]{Riley} and \cite[Question~3.1]{Arzhantseva};
	the case of $\operatorname{SL}_3(\mathbb Z)$ is also recorded in
	\cite[Question~7.11]{BHV08} and later discussed in
	\cite[Section~1.1]{Pham}.
	
   In this paper we settle this question in a more general form.
	
	\begin{theorem}\label{thm:main}
		Let $F$ be a field and let
		$\Gamma\leq\operatorname{GL}_n(F)$ be finitely generated and infinite.
		Then
		\[
		\kappa_{\mathrm u}(\Gamma)=0.
		\]
	\end{theorem}
	
	It is interesting to contrast \Cref{thm:main} with the main result of
	\cite{EMO}, which says that linear groups do have uniform exponential growth.
	
	As an immediate consequence of \Cref{thm:main}, no infinite finitely generated
	linear group has Property~$(T)$ uniformly. In particular,
	\[
	\kappa_{\mathrm u}\bigl(\operatorname{SL}_n(\mathbb Z)\bigr)=0
	\qquad (n\geq3).
	\]

	The initial argument of the proof is that if $\Gamma$ is mapped onto a product
	of $m$ different non-trivial finite groups, then $\Gamma$ has a finite generating set $S_m$ with $\kappa(\Gamma,S_m)<2/\sqrt{m}$.
	If this happens for every $m$, then $\kappa_u(\Gamma)=0$. This already
	suffices to settle the case of $\mathrm{SL}_n(\mathbb Z)$ as 
	$
	\mathrm{SL}_n(\mathbb Z)
    \twoheadrightarrow
	\prod_{i=1}^{m}\mathrm{SL}_n(\mathbb F_{p_i}),
	$
	where $\{p_1,\ldots,p_m\}$ are the first $m$ primes. The Lubotzky alternative for linear groups \cite[Window~9, Theorems~12 and~21]{LubotzkySegal03}
	implies that every non-virtually-solvable finitely generated linear group
	$\Gamma$ has a finite-index subgroup $\Lambda$ satisfying the above
	condition, and so
	$
	\kappa_u(\Lambda)=0.
	$
	We do not know if, in general,
	$
	\kappa_u(\Lambda)=0
	$
	for a finite-index subgroup $\Lambda$ of $\Gamma$ implies
	$
	\kappa_u(\Gamma)=0.
	$
	Still, we can show this for the current case (see  \Cref{lem:finite-extension}).
	
	Our proof (in contrast with non-uniform proofs of \cite{GelanderZuk} and \cite{Osin}) uses only
	finite representations; thus we are actually proving a stronger result. For this we need to recall the following definition.

   \begin{definition}
   	For a finite generating set $S$ of $\Gamma$, let $\tau(\Gamma,S)$ denote
   	the Kazhdan constant obtained by restricting to unitary representations
   	with finite image, and set
   	\[
   	\tau_{\mathrm u}(\Gamma)
   	=
   	\inf_{\substack{S\subseteq\Gamma\text{ finite}\\
   			\langle S\rangle=\Gamma}}
   	\tau(\Gamma,S).
   	\]
    \end{definition}
   	 
   	Since finite-image representations form a subclass of all unitary
   	representations, $\kappa_{\mathrm u}(\Gamma)\leq\tau_{\mathrm u}(\Gamma)$.
   	In fact, the proof of 	\Cref{thm:main} gives the stronger conclusion
   	\[
   	\tau_{\mathrm u}(\Gamma)=0.
   	\]
   In the non-virtually-solvable case, all the representations used in the
   proof factor through finite quotients.  In the virtually solvable case,
   the conclusion follows from amenability and residual finiteness, via
   the usual Følner-set argument in finite quotients (see~ \cite{LubotzkyWeiss}).

   There are groups without Property~$(T)$ and without nontrivial finite quotients, such groups have Property~$(\tau)$ (i.e. $\tau(\Gamma,S)>0$) but not Property~$(T)$. More interesting, and not
   trivial, is that there are finitely generated residually finite groups with Property~$(\tau)$ but not Property~$(T)$, e.g. $\operatorname{SL}_2(\mathbb Z[1/p])$ (see \cite[Example~4.3.3(E), p.~52]{Lubotzky}).
   But we don't know if there exist finitely generated residually finite groups
   with $\tau_u(\Gamma)>0$ while $\kappa_u(\Gamma)=0$.
   
   We finally mention that since $\kappa_u(\Gamma)\leq\kappa_u(\Delta)$
   (and $\tau_u(\Gamma)\leq\tau_u(\Delta)$) for every quotient $\Delta$ of $\Gamma$,
   the main results of this paper hold whenever $\Gamma$ has a finite-dimensional
   linear representation with infinite image. This applies, for example, to
   $\operatorname{Aut}(F_n)$, $n\geq4$, which has quite recently been proved to
   have $(T)$ \cite{Nitsche,KNO,KKN}, and so
   $\tau_u(\operatorname{Aut}(F_n))
   =\kappa_u(\operatorname{Aut}(F_n))=0$,
   in spite of $\operatorname{Aut}(F_n)$ being non-linear for $n\geq3$.

	The paper is organized as follows.  In Section~2 we establish the two
	Kazhdan-constant lemmas needed below.  Section~3 recalls the form of the
	Lubotzky alternative used in the proof.  In Section~4 we construct the finite
	quotients associated with the normal core, and Section~5 proves
	\Cref{thm:main}.

	
	\section{Quotients and finite extensions}
	
	We first record quotient monotonicity. 
	
	\begin{lemma}[Quotient monotonicity]\label{lem:quotient}
		Let $q:\Gamma\twoheadrightarrow\Lambda$ be an epimorphism of finitely
		generated groups.  Then
		\[
		\ku(\Gamma)\leq \ku(\Lambda).
		\]
	\end{lemma}
	
	\begin{proof}
		Fix a finite generating set $X=\{x_1,\ldots,x_d\}$ of $\Gamma$, and let
		$T=\{t_1,\ldots,t_r\}$ be an arbitrary finite generating set of $\Lambda$.
		Choose lifts $\widetilde t_i\in\Gamma$ of the $t_i$.  For each $j$, choose a
		word $w_j(T)$ representing $q(x_j)$ and put
		\[
		u_j=w_j(\widetilde t_1,\ldots,\widetilde t_r),
		\qquad
		r_j=x_ju_j^{-1}\in\ker q.
		\]
		Then
		\[
		S_T=\{\widetilde t_1,\ldots,\widetilde t_r,r_1,\ldots,r_d\}
		\]
		generates $\Gamma$, because $x_j=r_ju_j$.
		
		Let $\sigma:\Lambda\to\cU(\cH)$ be a unitary representation with
		$\cH^\Lambda=\{0\}$.  The pullback $\sigma\circ q$ has no nonzero
		$\Gamma$-invariant vector, each $r_j$ acts trivially, and the displacement by
		$\widetilde t_i$ equals the displacement by $t_i$.  Hence
		\[
		\kappa(\Gamma,S_T)\leq\kappa(\Lambda,T).
		\]
		Taking the infimum over all finite generating sets $T$ of $\Lambda$ proves the
		claim.
	\end{proof}

	\begin{lemma}[Factorwise finite extensions]\label{lem:finite-extension}
		Let
		\[
		1\longrightarrow M=M_1\times\cdots\times M_m
		\longrightarrow E\overset{q}{\longrightarrow}F
		\longrightarrow1
		\]
		be an exact sequence of finite groups.  Assume that $M_i\triangleleft E$ for
		every $i$.  Put
		\[
		N_i=\prod_{j\neq i}M_j,
		\qquad
		E_i=E/N_i,
		\]
		and regard $M_i$ as the normal subgroup $M_iN_i/N_i$ of $E_i$. Suppose that
		\begin{equation}\label{eq:frattini-condition}
			M_i\nleq\PhiG(E_i)
			\qquad (1\leq i\leq m),
		\end{equation}
		where $\PhiG(E_i)$ denotes the Frattini subgroup. Then
		\[
		\ku(E)\leq\frac{2}{\sqrt m}.
		\]
		Equivalently, for every $i$ it is enough to assume that there is a proper
		subgroup $H_i<E_i$ satisfying $E_i=M_iH_i$.
	\end{lemma}
	
	\begin{proof}
		Under the above identification,
		\[
		1\longrightarrow M_i\longrightarrow E_i\longrightarrow F\longrightarrow1
		\]
		is exact; the two hypotheses are equivalent.  Indeed, if $M_i\nleq\PhiG(E_i)$, choose a maximal subgroup $H_i<E_i$ which does not contain $M_i$.  Then $M_iH_i=E_i$.  Conversely, a proper supplement is contained in a maximal subgroup which cannot contain	$M_i$.
		
		Choose such a proper supplement $H_i<E_i$, and let
		$\widehat H_i$ be its full inverse image in $E$.  Thus
		\[
		N_i\leq\widehat H_i,
		\qquad
		E=M_i\widehat H_i.
		\]
		Let
		\[
		X_i=E/\widehat H_i\cong E_i/H_i,
		\qquad
		\cH_i=\ell^2_0(X_i).
		\]
		Since $H_i$ is proper, $|X_i|\geq2$.  Since $E=M_i\widehat H_i$, the group
		$M_i$ acts transitively on $X_i$.  Consequently
		\[
		\cH_i^{M_i}=\{0\}.
		\]
		Let $x_i=\widehat H_i\in X_i$ and define the unit vector
		\[
		\eta_i=
		\frac{\delta_{x_i}-|X_i|^{-1}\one_{X_i}}
		{\sqrt{1-|X_i|^{-1}}}
		\in\cH_i.
		\]
		It is fixed by $\widehat H_i$.
		
		We next construct a generating set of $E$ adapted simultaneously to the
		$\widehat H_i$. For each $i$, let $q_i:E_i\to F$ be the map induced by $q$,
		and write
		\[
		E_1\times_F\cdots\times_F E_m
		=
		\left\{
		(e_1,\ldots,e_m)\in E_1\times\cdots\times E_m:
		q_1(e_1)=\cdots=q_m(e_m)
		\right\}
		\]
		for the fiber product over $F$. Consider the natural map
		\[
		E\longrightarrow E_1\times_F\cdots\times_F E_m,
		\qquad
		e\longmapsto(eN_1,\ldots,eN_m).
		\]
		It is injective because $\bigcap_iN_i=1$. Moreover, each fiber of
		$q_i:E_i\to F$ has cardinality $|M_i|$, so
		\[
		\left|E_1\times_F\cdots\times_F E_m\right|
		=|F|\prod_{i=1}^m|M_i|=|E|.
		\]
		Hence the natural map is an isomorphism.
		
		Choose a generating set $R$ of $F$.  Since $H_i\to F$ is surjective, for every
		$r\in R$ choose $h_{i,r}\in H_i$ mapping to $r$.  By the fiber-product
		isomorphism, the tuple $(h_{1,r},\ldots,h_{m,r})$ has a lift $s_r\in E$.
		Then
		\[
		s_r\in\bigcap_{i=1}^m\widehat H_i,
		\]
		so every $s_r$ fixes every $\eta_i$.
		
		For each $i$, choose a finite generating set $T_i$ of $M_i$, and set
		\[
		S=\{s_r:r\in R\}\cup\bigcup_{i=1}^mT_i.
		\]
		The set $S$ generates $E$: its image generates $F$, and it contains generators
		of the kernel $M$.  Let
		\[
		\cH=\bigoplus_{i=1}^m\cH_i,
		\qquad
		\eta=\frac1{\sqrt m}(\eta_1,\ldots,\eta_m).
		\]
		The direct-sum permutation representation of $E$ on $\cH$ has no nonzero
		invariant vector, because $\cH_i^{M_i}=\{0\}$ for each $i$.
		
		Every $s_r$ fixes $\eta$.  If $t\in T_i$, then $t\in N_j\leq\widehat H_j$ for
		$j\neq i$, so $t$ fixes every component of $\eta$ except possibly the $i$th.
		Therefore
		\[
		\|t\eta-\eta\|
		=\frac1{\sqrt m}\|t\eta_i-\eta_i\|
		\leq\frac2{\sqrt m}.
		\]
		Thus $\kappa(E,S)\leq2/\sqrt m$, and hence
		$\ku(E)\leq2/\sqrt m$.
	\end{proof}
	
	We will use the following immediate special case.

	\begin{corollary}\label{cor:semisimple-kernel}
		In the setting of \Cref{lem:finite-extension}, suppose that assumption
		\eqref{eq:frattini-condition} is replaced by the assumption that every $M_i$
		is a nontrivial direct product of nonabelian finite simple groups. Then
		\[
		\ku(E)\leq\frac2{\sqrt m}.
		\]
	\end{corollary}

	\begin{proof}
		For completeness, recall that the Frattini subgroup of a finite group is
		nilpotent.  Indeed, let $P$ be a Sylow subgroup of $\PhiG(H)$.  The Frattini
		argument gives
		\[
		H=\PhiG(H)N_H(P).
		\]
		Since $\PhiG(H)$ consists of nongenerators, this implies $H=N_H(P)$.
		Hence every Sylow subgroup of $\PhiG(H)$ is normal, so $\PhiG(H)$ is
		nilpotent.
		
		If $M_i\leq\PhiG(E_i)$, then $M_i$ would be a subgroup of a finite nilpotent
		group and hence nilpotent.  This is impossible for a nontrivial direct product
		of nonabelian finite simple groups.  Thus the hypotheses of
		\Cref{lem:finite-extension} hold.
	\end{proof}
	
	\section{The arithmetic input}

	We isolate the precise arithmetic input needed below.

	\begin{proposition}
		\label{prop:lubotzky-product}
		Let $\Gamma$ be a finitely generated non-virtually-solvable linear group over
		an arbitrary field. Then there exist a finite-index subgroup
		$\Lambda\leq\Gamma$, pairwise nonisomorphic nonabelian finite simple groups
		\[
		S_1,S_2,\ldots,
		\]
		and a homomorphism
		\[
		\rho:\Lambda\longrightarrow\prod_{i=1}^{\infty}S_i
		\]
		with dense image. Equivalently, for every finite set $I\subset\mathbb N$, the
		coordinate map
		\[
		\rho_I:\Lambda\twoheadrightarrow\prod_{i\in I}S_i
		\]
		is surjective.
	\end{proposition}

	\begin{proof}
		Let $R$ be the finitely generated subring generated by the matrix entries of
		a finite generating set of $\Gamma$ and of the inverses of those generators,
		so that
		\[
		\Gamma\leq\operatorname{GL}_n(R).
		\]
		By the specialization theorem
		\cite[Proposition~8.1]{BCLM13}, there is a homomorphism from $R$ to a global
		field of the same characteristic such that the image of $\Gamma$ remains
		non-virtually-solvable.  Let
		\[
		\operatorname{sp}:\Gamma\twoheadrightarrow\overline{\Gamma}
		\]
		denote the resulting epimorphism onto the specialized image.
		
		Since $\overline{\Gamma}$ is not virtually soluble, the Lubotzky alternative
		\cite[Window~9, Theorems~12 and~21]{LubotzkySegal03} gives a finite-index
		subgroup
		\[
		\Lambda'\leq\overline{\Gamma},
		\]
		a global field $K$ (with $K=\mathbb Q$ in characteristic zero), a connected,
		simply connected simple algebraic group $\mathbf G$ over $K$, and a finite set
		of places $\Sigma$ such that there is a continuous epimorphism
		\[
		\widehat{\Lambda'}
		\twoheadrightarrow
		\mathbf G(\widehat{\mathcal O_\Sigma})
		=
		\prod_{v\notin\Sigma}\mathbf G(\mathcal O_v).
		\]
		
		Let
		\[
		\Lambda:=\operatorname{sp}^{-1}(\Lambda').
		\]
		Then $\Lambda$ has finite index in $\Gamma$, and the epimorphism
		$\Lambda\twoheadrightarrow\Lambda'$ induces a continuous epimorphism
		\[
		\widehat{\Lambda}\twoheadrightarrow\widehat{\Lambda'}.
		\]
		Hence
		\[
		\widehat{\Lambda}
		\twoheadrightarrow
		\prod_{v\notin\Sigma}\mathbf G(\mathcal O_v).
		\]
		
		After enlarging $\Sigma$ if necessary, reduction modulo $v$ is surjective
		for every $v\notin\Sigma$, and, apart from finitely many places $v$ with small
		residue field $\kappa(v)$, the group $\mathbf G(\kappa(v))$ has a nonabelian
		finite simple quotient $S_v$.
		Passing to an infinite subsequence of places, we may moreover arrange that
		the corresponding simple quotients have fixed Lie type and unbounded orders
		\cite[Window~2, Proposition~2]{LubotzkySegal03}.  Thus we may choose
		$v_1,v_2,\ldots$ so that
		\[
		|S_{v_1}|<|S_{v_2}|<\cdots.
		\]
		Put $S_i:=S_{v_i}.$ Then the groups $S_1,S_2,\ldots$ are pairwise nonisomorphic.
		
		Projecting to the factors indexed by $v_1,v_2,\ldots$ and then taking the
		corresponding finite simple quotients gives a continuous epimorphism
		\[
		\widehat{\Lambda}
		\twoheadrightarrow
		\prod_{i=1}^{\infty}S_i.
		\]
		Let
		\[
		\rho:\Lambda\longrightarrow\prod_{i=1}^{\infty}S_i
		\]
		be its restriction to $\Lambda$.  Since the canonical image of $\Lambda$ is
		dense in $\widehat{\Lambda}$, the image $\rho(\Lambda)$ is dense in
		$\prod_i S_i$.
		
		Finally, for every finite set $I\subset\mathbb N$, the projection of the dense
		subgroup $\rho(\Lambda)$ onto the finite discrete group
		$\prod_{i\in I}S_i$ is dense, and hence is the whole group.  Therefore
		\[
		\rho_I:\Lambda\twoheadrightarrow\prod_{i\in I}S_i
		\]
		is surjective.
	\end{proof}

	\section{Normal cores and semisimple blocks}
	
	We now prove the finite group-theoretic step that transfers the product
	quotients from a finite-index subgroup back to the ambient group.
	
	\begin{lemma}\label{lem:subdirect-simple}
		Let $S$ be a nonabelian finite simple group, and let
		$A\leq S^d$ project surjectively onto every coordinate.  Then
		\[
		A\cong S^a
		\]
		for some $1\leq a\leq d$.
	\end{lemma}
	
	\begin{proof}
		We argue by induction on $d$.  The case $d=1$ is immediate.  Let
		$p:A\to S^{d-1}$ be projection onto the first $d-1$ coordinates.  Its image
		$A'$ is subdirect, so by induction $A'\cong S^a$ for some $a$.

		Let $K$ be the image of $\ker p$ under the last-coordinate projection. Since
		$\ker p\triangleleft A$ and the last-coordinate projection of $A$ is $S$, we have
		$K\triangleleft S.$	As $S$ is simple, $K=1$ or $K=S$. Since
		$\ker p\leq \{1\}^{d-1}\times S$, these two cases correspond to
		$\ker p=1$ and $\ker p=\{1\}^{d-1}\times S$, respectively.
		
		If $\ker p=\{1\}^{d-1}\times S$, then
		$A=A'\times S\cong S^{a+1}$. If $\ker p=1$, then
		$p:A\to A'$ is an isomorphism, and hence $A\cong A'\cong S^a$.

	\end{proof}
	
	\begin{lemma}
		\label{lem:disjoint-subdirect}
		Let $S_1,\ldots,S_m$ be pairwise nonisomorphic nonabelian finite simple groups,
		and let $a_i\geq1$.  If
		\[
		A\leq\prod_{i=1}^m S_i^{a_i}
		\]
		projects surjectively onto every factor $S_i^{a_i}$, then
		\[
		A=\prod_{i=1}^m S_i^{a_i}.
		\]
	\end{lemma}
	
	\begin{proof}
		For each $i$, the surjection $A\twoheadrightarrow S_i^{a_i}$ shows that
		$S_i$ occurs at least $a_i$ times among the composition factors of $A$.
		Since the groups $S_1,\ldots,S_m$ are pairwise nonisomorphic, it follows that
		\[
		|A|\geq\prod_{i=1}^m |S_i|^{a_i}.
		\]
		On the other hand, since $A\leq\prod_{i=1}^m S_i^{a_i}$, the reverse
		inequality holds. Hence
		\[
		|A|=\prod_{i=1}^m |S_i|^{a_i},
		\]
		and therefore $A=\prod_{i=1}^m S_i^{a_i}$.
	\end{proof}

	\begin{proposition}
		\label{prop:core-blocks}
		Let $\Gamma_0\triangleleft\Gamma$ be a finite-index normal subgroup.  Suppose
		that for some $m$ there is an epimorphism
		\[
		\rho=(\rho_1,\ldots,\rho_m):
		\Gamma_0\twoheadrightarrow S_1\times\cdots\times S_m,
		\]
		where the $S_i$ are pairwise nonisomorphic nonabelian finite simple groups.
		Then $\Gamma$ has a finite quotient $E_m$ fitting into an exact sequence
		\begin{equation}\label{eq:core-extension}
			1\longrightarrow M_1\times\cdots\times M_m
			\longrightarrow E_m
			\longrightarrow \Gamma/\Gamma_0
			\longrightarrow1.
		\end{equation}
		where, for each $i$,
		\[
		M_i\cong S_i^{a_i}
		\qquad\text{for some }a_i\geq1,
		\]
		and $M_i\triangleleft E_m$.
	\end{proposition}
	
	\begin{proof}
		Put $K_i=\ker\rho_i$.  Since $K_i\triangleleft\Gamma_0$ and
		$\Gamma_0\triangleleft\Gamma$, its $\Gamma$-core
		\[
		C_i=\core_\Gamma(K_i)
		=\bigcap_{\gamma\in\Gamma}\gamma K_i\gamma^{-1}
		\]
		is the intersection of only finitely many conjugates, one for each coset of
		$\Gamma_0$ in $\Gamma$.  Thus $C_i\triangleleft\Gamma$ and
		$[\Gamma_0:C_i]<\infty$.
		
		Choose coset representatives $\gamma_1,\ldots,\gamma_d$ for
		$\Gamma/\Gamma_0$.  The natural map
		\[
		\Gamma_0/C_i\longrightarrow
		\prod_{r=1}^d \Gamma_0/(\gamma_rK_i\gamma_r^{-1})
		\cong S_i^d
		\]
		is injective and projects surjectively onto every coordinate.  By
		\Cref{lem:subdirect-simple},
		\[
		A_i:=\Gamma_0/C_i\cong S_i^{a_i}
		\]
		for some $a_i\geq1$.
		
		Set
		\[
		N=\bigcap_{i=1}^mC_i,
		\qquad
		E_m=\Gamma/N,
		\qquad
		M=\Gamma_0/N.
		\]
		The group $E_m$ is finite and $M\triangleleft E_m$.  The natural map
		\[
		M\longrightarrow\prod_{i=1}^mA_i
		\]
		is injective and subdirect.  By \Cref{lem:disjoint-subdirect}, it is an
		isomorphism.  Thus
		\[
		M\cong A_1\times\cdots\times A_m.
		\]
		For each $i$, define
		\[
		M_i=
		\left(\bigcap_{j\neq i}C_j\right)\big/N
		\leq M,
		\]
		where, when $m=1$, the empty intersection is understood to be $\Gamma_0$.
		Under the preceding product isomorphism, $M_i$ is exactly the $i$th factor
		$A_i$.  Since every $C_j$ is normal in $\Gamma$, each $M_i$ is normal in
		$E_m$.  Finally,
		\[
		E_m/M\cong\Gamma/\Gamma_0,
		\]
		which gives \eqref{eq:core-extension}.
	\end{proof}
	
	\section{Proof of the main theorem}
	
	\begin{proof}[Proof of \Cref{thm:main}]
		Suppose first that $\Gamma$ is virtually solvable.  Then $\Gamma$ is amenable.
		An amenable discrete group with property $(T)$ is finite
		\cite[Theorem~1.1.6]{BHV08}.  Since $\Gamma$ is infinite, it does not have
		property $(T)$, so $\kappa(\Gamma,S)=0$ for every finite generating set $S$.
		Consequently $\ku(\Gamma)=0$.
		
		Assume now that $\Gamma$ is not virtually solvable.  By
		\Cref{prop:lubotzky-product}, there are a finite-index subgroup
		$\Lambda\leq\Gamma$, pairwise nonisomorphic nonabelian finite simple groups
		$S_1,S_2,\ldots$, and a dense homomorphism
		\[
		\rho:\Lambda\longrightarrow\prod_{i=1}^{\infty}S_i.
		\]
		Let
		\[
		\Gamma_0=\core_\Gamma(\Lambda)
		=\bigcap_{\gamma\in\Gamma}\gamma\Lambda\gamma^{-1}.
		\]
		Then $\Gamma_0\triangleleft\Gamma$ and $[\Gamma:\Gamma_0]<\infty$.
		
		Let $P=\prod_{i=1}^{\infty}S_i$ and let
		$U=\overline{\rho(\Gamma_0)}$.  Since $\Gamma_0$ has finite index in
		$\Lambda$ and $\rho(\Lambda)$ is dense in $P$, the closed subgroup $U$ has
		finite index in $P$, hence is open.  Therefore there is a finite set
		$I_0\subset\mathbb N$ such that
		\[
		\prod_{i\notin I_0}S_i\leq U.
		\]
		After discarding and relabelling finitely many factors, it follows that, for
		every $m\geq1$, the coordinate map is surjective:
     	\begin{equation}\label{eq:simultaneous-simple}
		   \rho_m:\Gamma_0\twoheadrightarrow S_1\times\cdots\times S_m.
    	\end{equation}
		
		Apply \Cref{prop:core-blocks} to \eqref{eq:simultaneous-simple}.  We obtain a
		finite quotient $E_m$ of $\Gamma$ and an exact sequence
		\[
		1\longrightarrow M_1\times\cdots\times M_m
		\longrightarrow E_m
		\longrightarrow\Gamma/\Gamma_0
		\longrightarrow1,
		\]
		where every $M_i\triangleleft E_m$ is a nontrivial direct power of $S_i$.
		By \Cref{cor:semisimple-kernel},
		\[
		\ku(E_m)\leq\frac2{\sqrt m}.
		\]
		Since $E_m$ is a quotient of $\Gamma$, \Cref{lem:quotient} gives
		\[
		\ku(\Gamma)
		\leq\ku(E_m)
		\leq\frac2{\sqrt m}
		\qquad(m\geq1).
		\]
		Letting $m\to\infty$ yields $\ku(\Gamma)=0$.
	\end{proof}
	

	     \section*{Acknowledgments}
	         
	    ChatGPT was used as an exploratory tool in developing the direct-product argument that appeared in version 1 of this preprint; this argument was subsequently generalized to \Cref{lem:finite-extension}. J. Yao thanks Goulnara Arzhantseva and Narutaka Ozawa for their comments and suggestions on an earlier version of this preprint. A.L. was supported by the European Research Council (ERC) under the European Union's Horizon 2020 (N. 882751).


\begin{thebibliography}{99}
		
		
		\bibitem{Arzhantseva}
		G.~N. Arzhantseva,
		\emph{The uniform Kazhdan property for $\operatorname{SL}_n(\mathbb Z)$,
			$n\geq3$},
		Enseign. Math. (2) \textbf{54} (2008), no.~1--2, 11--12.
		
		\bibitem{BCLM13}
		E.~Breuillard, Y.~de Cornulier, A.~Lubotzky, and C.~Meiri,
		\emph{On conjugacy growth of linear groups},
		Math. Proc. Cambridge Philos. Soc. \textbf{154} (2013), no.~2, 261--277.
		
		\bibitem{BHV08}
		B.~Bekka, P.~de la Harpe, and A.~Valette,
		\emph{Kazhdan's Property $(T)$},
		New Mathematical Monographs, vol.~11, Cambridge University Press,
		Cambridge, 2008.
		
		\bibitem{EMO}
		A.~Eskin, S.~Mozes, and H.~Oh,
		\emph{On uniform exponential growth for linear groups},
		Invent. Math. \textbf{160} (2005), no.~1, 1--30.
		
		
		\bibitem{GelanderZuk}
		T.~Gelander and A.~\.{Z}uk,
		\emph{Dependence of Kazhdan constants on generating subsets},
		Israel J. Math. \textbf{129} (2002), 93--98.
		
		
		\bibitem{KKN}
		M.~Kaluba, D.~Kielak, and P.~W. Nowak,
		\emph{On property $(T)$ for $\operatorname{Aut}(F_n)$ and
			$\operatorname{SL}_n(\mathbb Z)$},
		Ann. of Math. (2) \textbf{193} (2021), no.~2, 539--562.
		
		\bibitem{KNO}
		M.~Kaluba, P.~W. Nowak, and N.~Ozawa,
		\emph{$\operatorname{Aut}(F_5)$ has property $(T)$},
		Math. Ann. \textbf{375} (2019), no.~3--4, 1169--1191.
		
		
		\bibitem{Kassabov}
		M.~Kassabov,
		\emph{Kazhdan constants for $\operatorname{SL}_n(\mathbb Z)$},
		Internat. J. Algebra Comput. \textbf{15} (2005), no.~5--6, 971--995.
		
		\bibitem{Lubotzky}
		A.~Lubotzky,
		\emph{Discrete Groups, Expanding Graphs and Invariant Measures},
		Progress in Mathematics, vol.~125, Birkh\"auser, Basel, 1994.
		
		\bibitem{LubotzkySegal03}
		A.~Lubotzky and D.~Segal,
		\emph{Subgroup Growth},
		Progress in Mathematics, vol.~212, Birkh\"auser, Basel, 2003.
		
		\bibitem{LubotzkyWeiss}
		A.~Lubotzky and B.~Weiss,
		\emph{Groups and expanders},
		in \emph{Expanding Graphs} (Princeton, NJ, 1992),
		DIMACS Ser. Discrete Math. Theoret. Comput. Sci., vol.~10,
		Amer. Math. Soc., Providence, RI, 1993, 95--109.
		
	   \bibitem{Nitsche}
	   M.~Nitsche,
	   \emph{Computer proofs for Property $(T)$, and SDP duality},
	   preprint, arXiv:2009.05134.
		
	
		
		\bibitem{Osin}
		D.~V. Osin,
		\emph{Kazhdan constants of hyperbolic groups},
		Funct. Anal. Appl. \textbf{36} (2002), no.~4, 290--297.
		
		\bibitem{OsinSonkin}
		D.~Osin and D.~Sonkin,
		\emph{Uniform Kazhdan groups},
		preprint, 2006, arXiv:math/0606012.
		
		\bibitem{Pham}
		L.~L. Pham,
		\emph{Uniform Kazhdan constants and paradoxes of the affine plane},
		Transform. Groups \textbf{27} (2022), no.~1, 239--269.
		
		
		\bibitem{Riley}
		T.~R. Riley,
		\emph{Navigating in the Cayley graphs of
			$\operatorname{SL}_N(\mathbb Z)$ and $\operatorname{SL}_N(\mathbb F_p)$},
		Geom. Dedicata \textbf{113} (2005), 215--229.
		
		\bibitem{Shalom}
		Y.~Shalom,
		\emph{Explicit Kazhdan constants for representations of semisimple and
			arithmetic groups},
		Ann. Inst. Fourier (Grenoble) \textbf{50} (2000), no.~3, 833--863.
		
		
		
		
		

		
	\end{thebibliography}
\end{document}